\documentclass[14pt]{amsart}
\usepackage{amsmath,amsfonts,amsthm}
\usepackage{amssymb,latexsym}
\usepackage[colorlinks]{hyperref}
\usepackage{graphicx}
\usepackage[all]{xy}
\usepackage{layout}
\begin{document}
\theoremstyle{plain}
\newtheorem{theorem}{Theorem}[section]
\newtheorem{lemma}[theorem]{Lemma}
\newtheorem{proposition}[theorem]{Proposition}
\newtheorem{corollary}[theorem]{Corollary}

\theoremstyle{definition}
\newtheorem{comment}{Comment}[section]
\newtheorem{definition}[theorem]{Definition}
\newtheorem{example}[theorem]{Example}
\newtheorem{counterexample}[theorem]{Counterexample}
\newtheorem{problem}[theorem]{Problem}
\newtheorem{note}[theorem]{Note}

\theoremstyle{remark}
\newtheorem*{notation}{Notation}
\newtheorem*{remark}{Remark}
\numberwithin{equation}{section}

\newcommand{\N}{\mathbb{N}}
\newcommand{\Z}{\mathbb{Z}}
\newcommand{\Q}{\mathbb{Q}}
\newcommand{\R}{\mathbb{R}}
\newcommand{\C}{\mathbb{C}}
\newcommand{\SP}{\: \: \: \: \:}
\title[$C^{*}$-algebras of Toeplitz and composition operators]{The $C^{*}$-algebra generated by irreducible Toeplitz and composition operators}

\author[M.S. Sarvestani, M. Amini]{Masoud Salehi Sarvestani, Massoud Amini}

\address[\textbf{Masoud Salehi Sarvestani}]{Department of Pure Mathematics\\ Faculty of Mathematical Sciences\\
 University of Tarbiat Modares\\
 Tehran\\ Iran}
\email{m.salehisarvestani@modares.ac.ir}

\address[\textbf{Massoud Amini}]{Department of Pure Mathematics\\ Faculty of Mathematical Sciences\\
 University of Tarbiat Modares\\
 Tehran\\ Iran}
\email {mamini@modares.ac.ir }


\begin{abstract}
We describe the $C^*$-algebra generated by an irreducible Toeplitz operator $T_{\psi}$, with continuous symbol $\psi $ on the unit circle $\mathbb{T}$, and finitely many composition operators on the Hardy space $H^2$ induced by certain  linear-fractional self-maps of the unit disc, modulo the ideal of compact operators $K(H^2)$. For composition operators with automorphism symbols, we show that this algebra is not isomorphic to the one generated by the shift and composition operators.
\end{abstract}

\maketitle
\section{Introduction}
\pagestyle{plain} The Hardy space $H^2=H^2(\mathbb{D})$ is the collection of all analytic functions $f$ on the open unit disk $\mathbb{D}$ satisfying the norm condition
$$\|f\|^2:=\lim_{r \rightarrow 1}\frac{1}{2\pi}\int_{0}^{2\pi}|f(re^{i\theta})|^2d\theta<\infty.$$
For any analytic self-map $\varphi$ of the open unit disk $\mathbb{D} $, a bounded composition operator on $H^2$ is defined by
$$C_\varphi :H^2\rightarrow H^2,~~~~C_\varphi (f)=f\circ\varphi.$$

If $f\in H^2$, then the radial limit
$f(e^{i\theta}):=\lim_{r\rightarrow1}f(re^{i\theta})$ exists almost
everywhere on the unit circle $\mathbb{T}$. Hence we can consider
$H^2$ as a subspace of $L^2(\mathbb{T})$. Let $\phi$ is a bounded
measurable function on $\mathbb{T}$ and $P_{H^2}$ be the
orthogonal projection of $L^2(\mathbb{T})$ (associated with
normalized arc-length measure on $\mathbb{T}$ ) onto $H^2$. The
Toeplitz operator $T_\phi$ is defined on $H^2$ by $T_\phi
f=P_{H_2}(\phi f)$ for all $f\in H^2$. Coburn in \cite{co67,co69}
shows that the unital $C^*$-algebra $C^*(T_z)$ generated by the unilateral
shift operator $T_z$ contains compact operators on $H^2$ as an ideal and every element $a\in C^*(T_z)$
has a unique representation $a=T_\phi +k$ for some $\phi \in
C(\mathbb{T})$ and $k\in \mathfrak K:=K(H^2)$. He shows that $C^*(T_z)/\mathfrak K$ is
$*$-isomorphic to $C(\mathbb{T})$, and determines
essential spectrum of Toeplitz operators with continuous symbol.

Recently the unital $C^*$-algebra
generated by the shift operator $T_z$ and the composition operator
$C_\varphi$ for a linear-fractional self-map $\varphi$ of
$\mathbb{D}$ is studied. For a linear-fractional self-map $\varphi$ on $ \mathbb{D}$, if
$\|\varphi\|_\infty <1$ then
$C_\varphi$ is a compact operator on $H^2$ \cite{sh93}. Therefore
one should consider those linear-fractional self-maps $\varphi$ which satisfy
$\|\varphi\|_\infty=1$. If moreover $\varphi$ is an  automorphism of
$\mathbb{D}$, then $C^*(T_z,C_\varphi)/\mathfrak K$ is
$*$-isomorphic to the crossed product $C(\mathbb{T})\rtimes_\varphi
\mathbb{Z}$ \cite {Ju07-1,Ju07-2}. When $\varphi$ is not an
automorphism there are three deferent cases:
\begin {itemize}
 \item [(i)]$\varphi$ has only one fixed point $\gamma$ which is on the unit circle $\mathbb{T}$ (i.e. $\varphi$ is a parabolic map). In this case $C^*(T_z,C_\varphi)/\mathfrak K$ is a commutative $C^*$-algebra isomorphic to the minimal unitization of $C_\gamma (\mathbb{T})\oplus C_0([0,1])$, where $C_\gamma (\mathbb{T})$ is the set of functions in $C (\mathbb{T})$ vanishing at $\gamma \in \mathbb{T}$ and $C_0([0,1])$ is the set of all $f \in C([0,1])$ vanishing at zero \cite{Qu13}.
 \item [(ii)] $\varphi$ has a fixed point $\gamma \in \mathbb T$ and fixes another point in $\mathbb{C}\cup \{\infty\}$ (equivalently $\varphi$ has a fixed point $\gamma \in \mathbb{T}$ and $\varphi^{'}(\gamma)\neq1$ ). In this case $C^*(T_z,C_\varphi)/\mathfrak K$ is $*$-isomorphic to the minimal unitization of $C_\gamma(\mathbb{T})\oplus(C_0([0,1])\rtimes \mathbb{Z})$ \cite{Qu13}.
 \item [(iii)]$\varphi$ fixes no point of $\mathbb{T}$ but there exist distinct points $\gamma,\eta \in \mathbb{T}$ with $\varphi(\gamma)=\eta$. In this case $C^*(T_z,C_\varphi)/\mathfrak K$ is the $C^*$-subalgebra $\mathcal{D}$ of $C(\mathbb{T})\oplus M_2(C([0,1]))$ defined by
     $$\mathcal{D}=\left\{(f,V)\in C(\mathbb{T})\oplus M_2(C([0,1])):V(0)=\left[
                                                                       \begin{array}{cc}
                                                                         f(\gamma) & 0 \\
                                                                         0 & f(\eta) \\
                                                                       \end{array}
                                                                     \right]
     \right\}$$ \cite {kmm07}.
 \end {itemize}

This paper generalizes the above results. The generalization is two fold. We replace the shift operator $T_z$ by an irreducible Toeplitz operator $T_{\psi}$ with continuous symbol $\psi $ on $\mathbb{T}$, and a single composition operator with finitely many composition operators on the Hardy space $H^2$ induced by certain linear-fractional self-maps of $\mathbb D$.
The paper is organized as follows. In section \ref{p} we review basic facts and known results which are used later in the paper. In section \ref{r}  we find the $C^*$-algebra
$C^*(T_z,C_{\varphi_1},...,C_{\varphi_n})/\mathfrak K$, where
$\varphi_1,...,\varphi_n$ are as in the case (ii). In section \ref{s} we replace the shift operator $T_z$ by an irreducible Toeplitz operator $T_\psi$ with symbol $\psi \in C(\mathbb{T})$ and obtain
more general results in the above cases. When $\varphi$ is an automorphism, the composition operator $C_{\varphi}$ often generates the unilateral shift operator $T_z$. We investigate this case in section \ref{t}.

\section{Preliminaries}\label{p}
In this section we review some of the known results which are used in the next sections. Here we use $C_0([0,1])$ to denote the set of functions in $C([0,1])$  vanishing at zero and $[T]$ to denote the
coset of operator $T\in B(H^2)$ in the Calkin algebra $B(H^2)/K(H^2)$.

A linear-fractional self-map $\rho$ with fixed pint $\gamma \in \mathbb{T}$ is {\it parabolic} if and only if $\rho^{'}(\gamma)=1$. In this case, $\rho$ is conjugate to a translation on the right half plan $\Omega :=\{z \in \mathbb{C}:Rez>0\}$ via the conformal map $\alpha:z\mapsto (\gamma+z)/(\gamma-z)$ of $\mathbb{D}$ onto $\Omega$. Therefor $\alpha \circ \rho \circ \alpha^{-1}$ is the translation map $z\mapsto z+a$ for some $a \in \mathbb{C}$ with non-negative real part. We denote the map $\rho$ by $\rho_{\gamma,a}$. This is an automorphism of $\mathbb{D}$ if and only if $Rea=0$. For $\gamma \in \mathbb{T}$, the set $\mathbb{P}_\gamma :=\{C_{\rho_{\gamma,a}}: Rea>0\}$ consists of all composition operators induced by parabolic non-automorphism self-maps of $\mathbb{D}$  fixing $\gamma$. If $\varphi$ is a parabolic non-automorphism self-map of $\mathbb{D}$, then $C_\varphi$ is irreducible \cite{mps}, and $C_\varphi^* C_\varphi-C_\varphi C_\varphi^*$ is a non-zero compact operator ($C_\varphi$ is essentially normal) \cite{blns03}. Therefore  the unital $C^*$-algebra $C^*(\mathbb{P}_\gamma)$ is irreducible and $C^*(\mathbb{P}_\gamma)\cap \mathfrak K\neq \{0\}$. By Theorem 2.4.9 in \cite {mu90}, $C^*(\mathbb{P}_\gamma)$ contains all compact operators on $H^2$. Since the elements of $\mathbb{P}_\gamma$ satisfy $\rho_{\gamma,a}\circ \rho_{\gamma,b}=\rho_{\gamma,a+b}$, $C^*(\mathbb{P}_\gamma)/\mathfrak K$ is a unital commutative $C^*$-algebra. The following theorem completely describes this $C^*$-algebra.
\begin{theorem} \cite[Theorem 3 and Corollary 2]{kmm10} \label {3,2}
There is a unique $*$-isomorphism $\Sigma : C([0,1])\rightarrow C^*(\mathbb{P}_\gamma)/ \mathfrak{K}$ such that $\Sigma (x^a)=[C_{\rho _{\gamma , a}}]$ for $Rea>0$. Moreover if $\rho$ is a parabolic non-automorphism self-map of $\mathbb{D}$ fixing $\gamma$, then
$C^*(C_\rho) = C^*(\mathbb{P}_\gamma)$.
\end{theorem}

For a linear-fractional self-map $\varphi$ of $\mathbb{D}$,
$$U_\varphi=C_\varphi(C_\varphi^* C_\varphi)^{-1/2}$$
is the partial isometry in the polar decomposition of $C_\varphi$. Since both  $C_\varphi$ and $C_\varphi^*$ are injective, $U_\varphi$ is a unitary operator.
The following  results are useful in finding the $C^*$-algebra generated by Toeplitz operators and composition operators induced by linear-fractional self-maps
of $\mathbb{D}$.
\begin{theorem}\cite [section 4]{kmm07} \label{3.1}
Suppose that $\varphi$ is a linear-fractional non-automorphism self-map
of $\mathbb{D}$  sending $\gamma\in\ \mathbb{T}$ to $\eta\in\
\mathbb{T}$. For every $f\in C(\mathbb{T})$ there exists compact
operators $k$ and $ k^{'}$ such that

$$T_{f} C_{\varphi}= f(\gamma)C_{\varphi}+k,~~~C_{\varphi} T_{f}=f(\eta)C_\varphi +k^{'}.$$

\end{theorem}
\begin{theorem} \cite{Ju07-1} \label {unitary}
Let $\varphi_1,\varphi_2$ are automorphisms of $\mathbb{D}$ and $f \in C(\mathbb{T})$. If $U_{\varphi_1}, U_{\varphi_2}$
are unitary parts of the polar decomposition of $C_{\varphi_1}, C_{\varphi_2}$, respectively, then the operators
$U_{\varphi_1}U_{\varphi_2}-U_{\varphi_2 o \varphi_1 }$ and $U_{\varphi_1}T_f U_{\varphi_1}^*-T_{f o \varphi_1 }$ are compact.
\end{theorem}
For $\gamma \in \mathbb{T}$ and non-negative real number $t$, following \cite{Qu13}, we consider the automorphism $\Psi_{\gamma,t}$ of $\mathbb{D}$ defined by
$$\Psi_{\gamma,t}(z)=\frac{(t+1)z+(1-t)}{(1-t)\bar{\gamma}z+(1+t)},$$
which fixes $\gamma$, and satisfies $\Psi_{\gamma,t}^{'}(z)=t$. The set $\{\Psi_{\gamma,t}:t>0\}$ is an abelian group as $\Psi_{\gamma,t_1}\circ \Psi_{\gamma,t_2}=\Psi_{\gamma,t_1t_2}$ for all $t_1,t_2>0$.

If $\varphi, \varphi _1 , \cdots ,\varphi _n$ are linear-fractional non-automorphism self-maps of $\mathbb{D}$ fixing $\gamma \in \mathbb{T}$, by equation (4.6) in \cite {Qu13}
\begin{equation} \label {a0}
 C^*(C_{ \varphi _ 1 },...,C_{ \varphi _ 1 }, \mathfrak K)=
 \end{equation}
 $$C^*(\{C_{\rho _{\gamma , a}} U_ {\Psi _{\gamma ,\varphi ^{'}_1(\gamma)^{m_1} ... \varphi ^{'}_n(\gamma)^{m_n}}}: Rea>0 ,(m_1,...,m_n) \in \mathbb{Z}^n \}).$$
In particular

\begin{equation} \label {a1}
 C^*(C_\varphi ,\mathfrak K)= C^*(\{C_{\rho _{\gamma , a}} U_ {\Psi _{\gamma ,\varphi ^{'}(\gamma)^n}}: Rea>0 ,n \in \mathbb{Z} \}).
\end{equation}

\begin{theorem} \cite[Theorem 4.4]{Qu13} \label {ca}
Let $G$ be a collection of automorphisms of $\mathbb{D}$ that fix $\gamma\in\mathbb{T}$. If $G$ is an abelian group and $\eta^{'}(\gamma)\neq1$ for all $\eta\in G\setminus \{id\}$ then $C^*(\{[C_{\rho _{\gamma , a}} U_ \eta]: Rea>0 ,\eta \in G \})$ is $*$-isomorphic to the minimal unitization of $C_0([0,1])\rtimes _\alpha G_d$ where the action $\alpha : G_d \rightarrow Aut(C_0([0,1]))$ is defined by $\alpha_\eta(f)(x)= f(x^{\eta^{'}(\gamma)})$ for $\eta \in G$, $f \in C_0([0,1])$, and $x \in [0,1]$.
\end{theorem}
By (\ref {a0}), (\ref {a1}) and Theorem \ref{ca}, the $C^*$-algebras $C^*(C_{\varphi_1},...,C_{\varphi_n},\mathfrak K)/\mathfrak K$ and $C^*(C_\varphi,\mathfrak K)/\mathfrak  K$ are determined as follows.
\begin{corollary} \cite[Theorems 4.6 and 4.7]{Qu13} \label {coro13}
Let $\varphi, \varphi _1 , \cdots ,\varphi _n$ are linear-fractional
non-automorphism self-maps of $\mathbb{D}$ that fix $\gamma \in
\mathbb{T}$, $\varphi^{'}(\gamma)\neq 1$ and $\ln \varphi ^{'}_1,...,\ln \varphi ^{'}_n$ are
linearly independent over  $\mathbb{Z}$. Define the actions $\alpha
: \mathbb{Z} \rightarrow Aut(C_0([0,1]))$ and $\alpha^{'} :
\mathbb{Z}^n \rightarrow Aut(C_0([0,1]))$ by
$\alpha_n(f)(x)=f(x^{\varphi^{'}(\gamma)^n})$ and $\alpha^{'}
_{(m_1,...,m_n)}(f)(x)$ $=f(x^{\varphi^{'}(\gamma)^{m_1}...\varphi^{'}(\gamma)^{m_n}})$,
respectively, for $f \in C_0([0,1])$, $n \in \mathbb{Z}$,
$(m_1,...,m_n) \in \mathbb{Z}^n$ and $x \in [0,1]$. Then
$C^{*}(C_\varphi,\mathfrak K)/\mathfrak K$ and
$C^{*}(C_{\varphi_1},...,C_{\varphi_n},\mathfrak K)/\mathfrak K$ are $*$-isomorphic to the minimal unitizations of $C_0([0,1])\rtimes_\alpha
\mathbb{Z}$ and $C_0([0,1])\rtimes_{\alpha^{'}} \mathbb{Z}^n$,
respectively.
\end{corollary}

\section{linear-fractional non-automorphism self-maps} \label{r}
Quertermous in \cite {Qu13} shows that if $\varphi$ is
a linear-fractional non-automorphism self-map of $\mathbb{D}$
fixing $\gamma \in \mathbb{T}$ then $C^{*}(T_z,C_\varphi)/\mathfrak K$ is
$*$-isomorphic to the minimal unitization of
$C_\gamma(\mathbb{T})\oplus C_0([0,1])$ where $C_\gamma(\mathbb{T})$
is the set of all $f \in C(\mathbb{T})$  vanishing at $\gamma$. We extend this result to finitely many composition
operators induced by linear-fractional non-automorphism
self-maps of $\mathbb{D}$ with a common fixed point on the unit
circle. Our approach is similar to that of Quertermous in
\cite{Qu13}, but there are some complications. As in previous section we use the notation $[T]$ for the coset of $T$ in the Calkin algebra. Let $t_1,..,t_n$ be nonzero positive real numbers, $\gamma \in \mathbb{T}$ and $\Sigma$ be the map defined in Theorem \ref {3,2}. Consider
$$\mathcal{N}_{\gamma,t_1,...,t_n}=\{\Sigma(g)[U_{\Psi_{\gamma,t_1^{m_1}...t_n^{m_n}}}]:g\in C_0([0,1]), (m_1,...,m_n)\in \mathbb{Z}^n\}$$
and let $\mathcal{A}_{\gamma,t_1,...,t_n}$ be the non-unital $C^*$-algebra generated by $\mathcal{N}_{\gamma,t_1,...,t_n}$. By Theorem \ref{3,2} and the fact that $\Psi_{\gamma,1}$ is the identity map of $\mathbb{D}$, $ \mathcal{A}_{\gamma,1,...,1}\cong C_0([0,1])$.

\begin{proposition} \label{3.1a}
If $\varphi_1,...,\varphi_n$ are linear-fractional non-automorphism self-maps of $\mathbb{D}$  fixing $\gamma \in \mathbb{T}$, then
$$C^*(T_z,C_{\varphi_1},...,C_{\varphi_n})/\mathfrak K=C^*(\{[T_f]:f \in C(\mathbb{T})\}\cup \mathcal{A}_{\gamma,\varphi_1^{'}(\gamma),...,\varphi_n^{'}(\gamma)}).$$
Moreover if $\ln\varphi^{'}_1(\gamma),...,\ln\varphi^{'}_n(\gamma)$ are linearly independent over $\mathbb{Z}$ then $$\mathcal{A}_{\gamma,\varphi^{'}_1(\gamma),...,\varphi^{'}_n(\gamma)}\cong C_0([0,1])\rtimes_{\alpha^{'}} \mathbb{Z}^n.$$
\end{proposition}
\begin{proof}
First we note that if $\varphi$ is a linear-fractional non-automorphism self-map of $\mathbb{D}$ fixing $\gamma \in \mathbb{D}$, then $\varphi^{'}(\gamma)>0$ (\cite{sh93}). By Theorem \ref{3,2}, $\Sigma(x^{a})=[C_{\rho_{\gamma,a}}]$ for $Rea>0$. Since the closed linear span of $\{x^a:Rea>0$ is dense in $C_0([0,1])$, $\mathcal{A}_{\gamma,\varphi^{'}_1(\gamma),...,\varphi^{'}_n(\gamma)}$ is the same as
\begin{align*}C^*(\{\Sigma(g)&[U_{\Psi_{\gamma,\varphi_1^{'}(\gamma)^{m_1}...\varphi_n^{'}(\gamma)^{m_n}}}]
:g\in C_0([0,1]), (m_1,...,m_n)\in \mathbb{Z}^n\})\\&=C^*(\{[C_{\rho_{\gamma ,a}}U_{\Psi_{\gamma,\varphi^{'}_1(\gamma)^{m_1}...\varphi^{'}_n(\gamma)^{m_n}}}]:Rea>0, (m_1,...,m_n)\in \mathbb{Z}^n\}).\end{align*}
Therefore by (\ref{a0}) $$C^*(T_z,C_{\varphi_1},...,C_{\varphi_n})/\mathfrak K=C^*(\{[T_f]:f \in C(\mathbb{T})\}\cup \mathcal{A}_{\gamma,\varphi_1^{'}(\gamma),...,\varphi_n^{'}(\gamma)}).$$ The last statement is a consequence of Corollary \ref{coro13}.\end{proof}

We set
$$\mathcal{C}_{\gamma,t_1,...,t_n}=\{[T_f]+A:f \in C^*(\psi), A\in \mathcal{A}_{\gamma,t_1,...,t_n}\}.$$
By Proposition \ref{3.1a}, $C^*(T_z,C_{\varphi_1},...,C_{\varphi_n})/\mathfrak K=C^*(\mathcal{C}_{\gamma,\varphi_1^{'}(\gamma),...,\varphi_n^{'}(\gamma)})$. We show that $\mathcal{C}_{\gamma,\varphi_1^{'}(\gamma),...,\varphi_n^{'}(\gamma)}$ is indeed a $C^*$-algebra and describe it. We need the following Lemma.
\begin{lemma} \cite [Lemma 6.3]  {Qu13} \label{lem}
If $\gamma \in \mathbb{T}$, $f\in C(\mathbb{T})$ and $A\in \mathcal{A}_{\gamma,t_1,...,t_n}$, then
$$[T_f]A=f(\gamma)A=A[T_f].$$
Moreover if $[T_f]+A=[0]$ then $f\equiv0$ and $A=0.$
\end{lemma}
The first part of above lemma says that $\mathcal{C}_{\gamma,\varphi_1^{'}(\gamma),...,\varphi_n^{'}(\gamma)}$ is closed under multiplication and it is a dense $*$-subalgebra of $C^*(T_z,C_{\varphi_1},...,C_{\varphi_n})/\mathfrak K$.
\begin{theorem} \label{unitization}
If $ \varphi _1,...,\varphi_n $ are linear-fractional non-automorphism self-maps of $\mathbb{D}$  fixing $\gamma \in \mathbb{T}$, then $C^*(T_z,C_{\varphi _1},...,C_{\varphi _1})/\mathfrak K $ is $*$-isomorphic to the minimal unitization of $C_\gamma (\mathbb{T})\oplus \mathcal{A}_{\gamma,\varphi_1^{'}(\gamma),...,\varphi_n^{'}(\gamma)}$.
\begin{proof}
Let $\mathcal{B}$ is the minimal unitization of $C_\gamma (\mathbb{T})\oplus \mathcal{A}_{\gamma,\varphi_1^{'}(\gamma),...,\varphi_n^{'}(\gamma)}$. Since $C_\gamma (\mathbb{T})=\{f-f(\gamma):f\in C(\mathbb{T})\}$ we may define the map $\Delta :\mathcal{B}\rightarrow \mathcal{C}_{\gamma,\varphi_1^{'}(\gamma),...,\varphi_n^{'}(\gamma)}$ by
$$\Delta((f-f(\gamma),A)+f(\gamma)I)=[T_f]+A,$$
for $f\in C(T)$ and $A\in \mathcal{A}_{\gamma,\varphi_1^{'}(\gamma),...,\varphi_n^{'}(\gamma)} $. By part two of Lemma \ref{lem}, $\Delta$ is injective.
Let $\alpha=(f_1-f_1(\gamma),A_1)+f_1(\gamma) I$ and $\beta=(f_2-f_2(\gamma),A_2)+f_2(\gamma) I$, for some $f_1,f_2 \in C(T)$ and $A_1,A_2 \in \mathcal{A}$. Then
\begin{align*}\Delta(\alpha \beta)&=\Delta(f_1f_2-f_1(\gamma)f_2-f_2(\gamma)f_1+f_1(\gamma)f_2(\gamma),A_1A_2)\\
&+f_1(\gamma)(f_2-f_2(\gamma),A_2)+f_2(\gamma)(f_1-f_1(\gamma),A_1)+f_1(\gamma)f_2(\gamma)\\
&=\Delta(f_1f_2-f_1(\gamma)f_2(\gamma),A_1A_2+f_1(\gamma)A_1+f_2(\gamma)A_2)+f_1(\gamma)f_2(\gamma)I\\
&=[T_{f_1f_2}]+A_1A_2+f_1(\gamma)A_2+f_2(\gamma)A_1.\end{align*}
Therefore, again by Lemma \ref{lem},
$$\Delta(\alpha \beta)=[T_{f_1f_2}]+A_1A_2+[T_{f_1}]A_2+[T_{f_2}]A_1=([T_{f_1}]+A_1)([T_{f_2}]+A_2)=\Delta(\alpha)\Delta(\beta).$$
Hence $\Delta$ is an injective $*$-homomorphism, and
thus an isometry with closed range $\mathcal{C}_{\gamma,\varphi_1^{'}(\gamma),...,\varphi_n^{'}(\gamma)}$. Therefore the $C^*$-algebra $$C^*(T_z,C_{\varphi _1},...,C_{\varphi _1})/\mathfrak K =\mathcal{C}_{\gamma,\varphi_1^{'}(\gamma),...,\varphi_n^{'}(\gamma)}$$ is $*$-isomorphic to the minimal unitization of $C_\gamma (\mathbb{T})\oplus \mathcal{A}_{\gamma,\varphi_1^{'}(\gamma),...,\varphi_n^{'}(\gamma)}$.
\end{proof}
\end{theorem}
The following result immediately follows from the above theorem and Proposition \ref{3.1a}.

\begin{corollary}
If $ \varphi _1,...,\varphi_n $ are linear-fractional non-automorphism self-maps of $\mathbb{D}$ fixing $\gamma \in \mathbb{T}$ and $\ln \varphi_1^{'}(\gamma),...,\ln \varphi_n^{'}(\gamma)$ are linearly independent over $\mathbb{Z}$,
then $C^*(T_z,C_{\varphi _1},...,C_{\varphi _n})/\mathfrak K $ $*$-isomorphic to the minimal unitization of the direct sum $C_\gamma(\mathbb{T})\oplus( C_0([0,1])\rtimes_{\alpha^{'}} \mathbb{Z}^{n})$.
\end{corollary}

\section{Irreducible Toeplitz operators } \label{s}
Let $X$ be a compact Hausdorff space and $\mathcal{A}$ be a $C^*$-subalgebra of $C(X)$ containing the constants. For $x,y \in X$, put $x\thicksim y$ if and only if $f(x)=f(y)$ for all $f$ in $\mathcal{A}$. Then $\thicksim$ is an equivalence relation on $X$. Let $[x]$ denote the equivalence class of $x$ and $[X]$ be the quotient space and equip $[X]$ with the weak topology induced by $\mathcal{A}$. Let $X/\thicksim$ be the quotient space equipped with the quotient topology. Then $\mathcal{A}$ is $*$-isomorphic to $C([X])$ and a $C^*$-subalgebra of $C(X/\thicksim)$ via $f\mapsto \tilde{f}$ where $\tilde{f}([x]):=f(x)$ for $x \in X$. Note that $[X]$ is always Hausdorff and it is homeomorphic to $X/\thicksim$ when the latter is Hausdorff \cite{co69}.

If $D$ is an irreducible $C^*$-subalgebra of $C^*(T_z)$, then it contains a nonzero compact operator \cite{co69}. Hence $D$ contains all of compact operators on $H^2$ \cite{mu90}. We set
$$D_0= \{f \in C(\mathbb{T}): T_f \in D \}.$$
\begin{theorem}\cite [Theorem 2 and 5]  {co69} \label{1co69}
If $D$ is an irreducible $C^*$-subalgebra of $C^*(T_z)$, then $D_0$ is a $C^*$-subalgebra of $C(\mathbb{T})$ and $D/\mathfrak K$ is $*$-isomorphic to $C([\mathbb{T}])$, where $[\mathbb{T}]$ is the quotient with respect to the equivalence relation induced by $D_0$.
\end{theorem}

Let $T_\psi$ be an irreducible Toeplitz operator (i.e. the only closed vector subspaces of $H^2$ reducing for $T_{\psi}$ are $0$ and $H^2$) with continuous symbol $\psi$ (see \cite{bh64,co67,co69,no67}). Then $D=C^*(T_\psi)$ is irreducible and by Theorem 6 in \cite{co69}, $D_0$ is the $C^*$-algebra of $C(\mathbb{T})$ generated by $\psi$. Hence by Theorem \ref{1co69}, $D/\mathfrak K$ is $*$-isomorphic to $C([\mathbb{T}])$ where $[\mathbb{T}]$ is the quotient with respect to the equivalence relation induced by $\psi$ (that is $x\thicksim y$ if and only if $\psi(x)=\psi(y)$).

Note that $T_z$ is irreducible and there are other irreducible Toeplitz operators (for example, see Example 1 and 2 in \cite{no67}). If $D=C^*(T_{\psi})= C^*(T_z)$ For some continuous function $\psi$,
then $D_0=C(\mathbb{T})$ is generated by $\psi$ and by the Stone-Weierstrass theorem, $\psi$ must be one-to-one on the unit circle. Therefore we are interested in the case that $\psi$ is not one-to-one on $\mathbb{T}$.

Let $T_\psi$ be an irreducible Toeplitz operator and  $[\mathbb T]$ be the quotient space with respect to the equivalence relation induced by $\psi$. Put
$$C_{[\gamma]}([\mathbb{T}]):=\{f\in C([\mathbb{T}]):f([\gamma])=0\}$$
and let $\mathcal{B}_{\gamma,t}$ be the~minimal unitization of $C_{[\gamma]}([\mathbb{T}])\oplus \mathcal{A}_{\gamma,t}. $

\begin{theorem} \label{irre}
If $T_\psi$ is an irreducible Toeplitz operator on Hardy space $H^2$ with symbol $\psi \in C(\mathbb{T})$ and $\varphi$ is a linear-fractional non-automorphism self-map of $\mathbb{T}$  fixing $\gamma \in \mathbb{T}$, then $C^*(T_\psi,C_\varphi)/\mathfrak K$ is $*$-isomorphic to $\mathcal{B}_{\gamma,\varphi^{'}(\gamma)}$.
\begin{proof}
Similar to the proof of Proposition \ref{3.1a},
\begin{equation} \label {b1}
C^*(T_\psi,C_\varphi)/\mathfrak K=C^{*}(\{[T_\phi]:\varphi \in C^{*}(\psi)\}\cup\mathcal{A}_{\gamma,\varphi^{'}(\gamma)}).
\end{equation}
For $t>0$, set
$$\mathcal{C}_{\gamma,t,\psi}:=\{[T_\phi]+A:\phi \in C^*(\psi), A\in \mathcal{A}_{\gamma,t}\}.$$
By (\ref{b1}), $C^*(T_\psi,C_\varphi)/\mathfrak K=C^*(C_{\gamma,t,\psi})$. We show that $C_{\gamma,t,\psi}$ is a $C^*$-algebra and describe it. It is clear that $C_{\gamma,t,\psi}$ is closed under taking linear combination and adjoint. On the other hand, by Lemma \ref{lem} and the fact that for $\phi_1,\phi_2 \in C^*(\psi)$, $[T_{\phi_1}][T_{\phi_2}]=[T_{\phi_1\phi_2}]$ (since $T_{\phi_1\phi_2}-T_{\phi_1}T_{\phi_2}$ is a compact operator), $C_{\gamma,t,\psi}$ is also closed under multiplication. Hence $\mathcal{C}_{\gamma,t,\psi}$ is a dense $*$-subalgebra of $C^*(T_\psi,C_\varphi)/\mathfrak K$. Similar to the proof of Theorem \ref{unitization}, we define the map $\mathcal{F}:\mathcal{B}_{\gamma,\varphi^{'}(\gamma)}\rightarrow \mathcal{C}_{\gamma,\varphi^{'}(\gamma),\psi}$, by
$$\mathcal{F}((\tilde{f}-\tilde{f}([\gamma]),A)+\tilde{f}([\gamma])I=[T_f]+A,$$
for $f\in C^*(\psi)$ and $A\in \mathcal{A}_{\gamma,\varphi^{'}(\gamma)}$. By Lemma \ref{lem}, $\mathcal{\mathcal{F}}$ is an injective $*$-homomorphism and its image is $\mathcal{C}_{\gamma,\varphi^{'}(\gamma),\psi}=C^*(T_\psi,C_\varphi)/\mathfrak K$.
\end{proof}
\end{theorem}
The following results are straightforward consequences of Theorems \ref{3,2}, \ref{irre} and Corollary \ref{coro13}.
\begin{corollary}
If $T_\psi$ is irreducible with symbol $\psi $ in $C(\mathbb{T})$ and $\rho$ is a parabolic non-automorphism self-map of $\mathbb{D}$ fixing $\gamma \in \mathbb{T}$ then,
$C^*(T_\psi,C_\varphi)/\mathfrak K$ is $*$-isomorphic to the minimal unitization of $C_{[\gamma]}([\mathbb{T}])\oplus C_0([0,1])$.
\end{corollary}
\begin{corollary}
If $T_\psi$ is irreducible with symbol $\psi $ in $C(\mathbb{T})$ and $\varphi$ is linear-fractional non-automorphism self-map of $\mathbb{D}$  fixing $\gamma \in \mathbb{T}$ such that $\varphi^{'}(\gamma)\neq1$, then
$C^*(T_\psi,C_\varphi)/\mathfrak K$ is $*$-isomorphic to the minimal unitization of $C_{[\gamma]}([\mathbb{T}])\oplus (C_0([0,1])\rtimes_\alpha\mathbb{Z})$,
where the action $\alpha$ is defined as in Corollary \ref {coro13}.
\end{corollary}

\begin{corollary}
If $T_\psi$ is irreducible with symbol $\psi$ in $C(\mathbb{T})$ and $ \varphi _1,...,\varphi_n $ are linear-fractional non-automorphism self-maps of $\mathbb{D}$ fixing $\gamma \in \mathbb{T}$ such that $\ln \varphi_1^{'}(\gamma),...,\ln \varphi_n^{'}(\gamma)$ are linearly independent over $\mathbb{Z}$,
then $C^*(T_\psi,C_{\varphi _1},...,C_{\varphi _n})/\mathfrak K $ is $*$-isomorphic to the minimal unitization of $C_{[\gamma]}([\mathbb{T}])\oplus( C_0([0,1])\rtimes_{\alpha^{'}} \mathbb{Z}^{n})$.
\end{corollary}

Now consider the case that $\varphi$ is a linear-fractional non-automorphism self-map of $\mathbb{D}$ such that $\varphi(\gamma)=\eta$ for some $\varphi\neq\eta \in \mathbb{T}$. Kriete, MacCluer and Moorhouse investigated this case in \cite{kmm07}. We summarize their results as follows.
\begin{theorem} \cite{kmm07} \label {kmm,1}
Let $\varphi$ be a linear-fractional non-automorphism self-map of $\mathbb{D}$ with $\varphi(\gamma)=\eta$ for some $\varphi\neq\eta \in \mathbb{T}$. Then for every $a\in \mathcal{A}=C^*(T_z,C_\varphi)/\mathfrak K$ there is a unique $\omega\in C(\mathbb{T})$ and unique functions $f,g,h$ and $k$ in $C_0([0,1])$ such that
$$a=[T_\omega]+f([C_\varphi^* C_\varphi])+g([C_\varphi C_\varphi^*])+[U_\varphi]h([C_\varphi^* C_\varphi])+[U_\varphi^*]k([C_\varphi C_\varphi^*]).$$
Moreover the map $\Phi : \mathcal{A}\rightarrow C(\mathbb{T})\oplus \mathbb{M}_2(C([0,1]))$ defined by
$$\Phi(a)=\left(\omega,\left[
                   \begin{array}{cc}
                     \omega(\gamma)+g & h \\
                     k & \omega(\eta)+f \\
                   \end{array}
                 \right]
\right) $$
is a $*$-isomorphism of $\mathcal{A}$ onto the following $C^*$-subalgebra of $C(\mathbb{T})\oplus \mathbb{M}_2(C([0,1]))$
$$\mathcal{D}=\left\{(\omega,V)\in C(\mathbb{T})\oplus \mathbb{M}_2(C([0,1])):V(0)=\left[
                                                                       \begin{array}{cc}
                                                                         \omega(\gamma) & 0 \\
                                                                         0 & \omega(\eta) \\
                                                                       \end{array}
                                                                     \right]
     \right\}.$$
\end{theorem}
We replace the shift operator by an arbitrary irreducible Toeplitz operator $T_\psi$ with continuous symbol.
\begin{theorem}
Let $\varphi$ be a linear-fractional non-automorphism self-map of $\mathbb{D}$ such that $\varphi(\gamma)=\eta$ for distinct points $\gamma,\eta\in \mathbb{T} $ and $T_\psi$ be irreducible with continuous symbol $\psi$ on $\mathbb{T}$. Then every element $b$ in $\mathcal{B}=C^*(T_\psi,C_\varphi)/\mathfrak K$ has a unique representation of the form
$$b=[T_\omega]+f([C_\varphi^* C_\varphi])+g([C_\varphi C_\varphi^*])+[U_\varphi]h([C_\varphi^* C_\varphi])+[U_\varphi^*]k([C_\varphi C_\varphi^*])$$
where $\omega\in C^*(\psi)$ and $f,g,h$ and $k$ are in $C_0([0,1])$. Moreover $\mathcal{B}$ is $*$-isomorphic to the $C^*$-subalgebra $\mathcal{D}$ of $C([\mathbb{T}])\oplus M_2(C([0,1]))$ defined by
$$\mathcal{D}=\left\{(f,S)\in C([\mathbb{T}])\oplus \mathbb{M}_2(C([0,1])): S(0)=\left[
                                                            \begin{array}{cc}
                                                              f([\gamma]) & 0 \\
                                                              0 & f([\eta]) \\
                                                            \end{array}
                                                          \right]
\right\}.$$
\begin{proof}
Since $\mathcal{B}$ is a $C^*$-subalgebra of $C^*(T_z,C_\varphi)/\mathfrak K$, by Theorem \ref{kmm,1}, for every element $b\in \mathcal{B}$ there is a unique $\omega\in C(\mathbb{T})$ and unique functions $f,g,h$ and $k$ in $C_0([0,1])$ such that
$$b=[T_\omega]+f([C_\varphi^* C_\varphi])+g([C_\varphi C_\varphi^*])+[U_\varphi]h([C_\varphi^* C_\varphi])+[U_\varphi^*]k([C_\varphi C_\varphi^*]).$$
We show that $\omega \in C^*(\psi)$. By Theorem \ref{3.1}, for each $\varepsilon >0$, there is an element $b_\varepsilon\in \mathcal{B}$ such that $\|b_\varepsilon-b\|<\varepsilon$ and $b_\varepsilon=p([T_\psi],[T_\psi^*])+q([C_\varphi],[C_\varphi^*])$, for some polynomials $p$ and $q$. It is straightforward to show that $p([T_\psi],[T_\psi^*])=[T_{p(\psi,\bar{\psi})}]$. Hence by Theorem \ref{kmm,1}, there are unique functions $f_\varepsilon,g_\varepsilon,h_\varepsilon$ and $k_\varepsilon$ in $C_0([0,1])$ such that
$$b_\varepsilon-b=[T_{\omega-p(\psi.\bar{\psi})}]+f_\varepsilon([C_\varphi^* C_\varphi])+g_\varepsilon([C_\varphi C_\varphi^*])+[U_\varphi]h_\varepsilon([C_\varphi^* C_\varphi])$$$$+[U_\varphi^*]k_\varepsilon([C_\varphi C_\varphi^*])$$
and
$$\left\|\left(\omega-p(\psi,\bar{\psi}),\left[
                                 \begin{array}{cc}
                                   \omega(\gamma)+g_\varepsilon & h_\varepsilon \\
                                   k_\varepsilon & \omega(\eta)+f_\varepsilon \\
                                 \end{array}
                               \right]
\right)\right\|=\|b_\varepsilon-b\|<\varepsilon.$$
Thus $\|\omega-p(\psi,\bar{\psi})\|<\varepsilon$ and $\omega \in C^*(\psi)$. By Theorem \ref{kmm,1}, $\mathcal{B}$ is $*$-isomorphic to the following $C^*$-subalgebra of $C^*(\psi)\oplus M_2(C([0,1]))$
$$\mathcal{C}=\left\{(\omega,S)\in C^*(\psi)\oplus \mathbb{M}_2(C([0,1])): S(0)=\left[
                                                            \begin{array}{cc}
                                                              \omega(\gamma) & 0 \\
                                                              0 & \omega(\eta) \\
                                                            \end{array}
                                                          \right]
\right\}.$$
Hence by Theorem \ref{1co69}, $\mathcal{B}$ is $*$-isomorphic to
$$\mathcal{D}=\left\{(f,S)\in C([\mathbb{T}])\oplus \mathbb{M}_2(C([0,1])): S(0)=\left[
                                                            \begin{array}{cc}
                                                              f([\gamma]) & 0 \\
                                                              0 & f([\eta]) \\
                                                            \end{array}
                                                          \right]
\right\}.$$
\end{proof}

\end{theorem}

\section{composition operators with automorphism symbols} \label{t}
A self-map $\varphi$ of the unit disk $\mathbb{D}$ is an automorphism if $\varphi$ is a one-to-one holomorphic  map of $\mathbb{D}$ onto
$\mathbb{D}$. We denote the class of automorphisms of $\mathbb{D}$ by $Aut(\mathbb{D})$. A well-know consequence of Schwarz Lemma shows
that every element $\varphi \in Aut(\mathbb{D})$ has the form
\begin{equation} \label {Auto}
\varphi(z)=\omega \frac{s-z}{1-\bar{s}z},
\end{equation}
for some $\omega \in \mathbb{T}$, where $s=\varphi^{-1}(0)\in \mathbb{D}$.

A {\it Fuchsian group} $\Gamma$ is a discrete group of automorphisms of $\mathbb{D}$. Fix a point $z_{0}\in \mathbb{D}$. The limit set of
$\Gamma$ is the set of limit points of the orbit $\{\varphi(z_0):\varphi \in \Gamma \}$ in $\mathbb{D}$. This
is a closed subset of the unit circle and does not depend on the choice of $z_0$. The limit set of a Fuchsian group has either 0,1,2, or infinitely
many elements. When the limit set is infinite, it is  perfect and nowhere dense (hence uncountable). A Fuchsian group $\Gamma$ is called {\it non-elementary} if its limit set is infinite.

Jury in \cite{Ju07-1} describes the $C^*$-algebra $C^*(\{C_{\varphi}:\varphi \in \Gamma\})/\mathfrak K$ when $\Gamma$ is a non-elementary Fuchsian group.
A basic point in the proof is that the non-elementary condition on $\Gamma$ guarantees that the $C^*$-algebra $C^*(\{C_{\varphi}:\varphi \in \Gamma\})$ contains the unilateral shift $T_z$. In the next proposition we weaken
this condition on  $\Gamma$ and still get the shift operator. We need the following lemma.
\begin{lemma} \cite [Lemma 3.2]{Ju07-1}\label{Jury lemma}
If $\varphi$ is an automorphism of $\mathbb{D}$ with $a=\varphi^{-1}(0)$ and
$$f(z)=\frac{1-\bar{a}z}{(1-|a|^2)^{1/2}},$$
then $C_\varphi C_\varphi^*=T_zT_z^*$.
\end{lemma}
\begin{proposition} \label{two points}
Let $\Gamma$ be a group of automorphisms on $\mathbb{D}$. If the orbit $\{\varphi(0):\varphi \in \Gamma \}$ consists of at least two linear independent points over $\mathbb{R}$, then $C^*(\{C_{\varphi}:\varphi \in \Gamma\})$ contains the unilateral shift $T_z$.
\end{proposition}
\begin{proof} Let $\varphi_{1},\varphi_{2}$ are two distinct elements of $\Gamma$. If $a_{1}=\varphi_{1}(0)$ and $a_{2}=\varphi_{2}(0)$, then by Lemma \ref{Jury lemma},
$$(1-|a_{i}|^2)C_{\varphi_{i}^{-1}}C_{\varphi_{i}^{-1}}^*=(I-\bar{a_{i}}T_z)(I-\bar{a_{i}}T_z)^*,$$
for $i=1,2$. Therefore the unital $C^*$-algebra $\mathcal{D}:=C^*(\{C_{\varphi}:\varphi \in \Gamma\})$ contains
$$s_{i}=\bar{a_i}T_z+a_iT_z^*-|\bar{a_i}|^{2}T_zT_z^*,~~~~~(i=1,2).$$
If $a_1=\alpha +i\beta, a_2=\alpha^{'}+i\beta^{'}$ then a simple computation shows that
$$det\left(
       \begin{array}{cc}
         a_1 & a_2 \\
         \bar{a_1} & \bar{a_2} \\
       \end{array}
     \right)=-2i~~det
     \left(
       \begin{array}{cc}
         \alpha & \alpha^{'} \\
         \beta & \beta^{'} \\
       \end{array}
     \right).
$$
Hence if $a_1,a_2$ are linearly independent over $\mathbb{R}$ then the linear system
$$\left(
    \begin{array}{cc}
      a_1 & a_2 \\
         \bar{a_1} & \bar{a_2} \\
    \end{array}
  \right)
  \left(
    \begin{array}{c}
      x \\
      y \\
    \end{array}
  \right)=
  \left(
    \begin{array}{c}
      1 \\
      0 \\
    \end{array}
  \right)
$$
has a unique solution. Therefore  a linear combination of $s_1,s_2$ is of the form $T_z^*+tT_zT_z^*$,
for some scaler $t \in \mathbb{C}$. If $t=0$, then $\mathcal{D}$ contains $T_z$, otherwise, since
$$I+tT_z=(T_z^*+tT_zT_z^*)(T_z^*+tT_zT_z^*)^*-\bar{t}(T_z^*+tT_zT_z^*) \in \mathcal{D},$$
again $T_z \in \mathcal{D} $.
\end{proof}

\begin{corollary} \label{}
Let $\Gamma$ be a Fuchsion group. If the limit set of $\Gamma$ contains at least two linearly independent points, then
$C^*(\{C_{\varphi}:\varphi \in \Gamma\})$ contains the unilateral shift operator.
\end{corollary}
\begin{proof}
If $z_1,z_2$ are linearly independent points of the limit set, then one can choose two elements of the orbit $\{\varphi(0):\varphi \in \Gamma\}$ near to
$z_1,z_2$ that are linearly independent. Hence by Proposition \ref{two points}, the above $C^*$-algebra contains the unilateral shift operator.
\end{proof}

As a consequence if $\Gamma$ be a non-elementary Fuchsion group, then the above corollary shows that $C^*(\{C_{\varphi}:\varphi \in \Gamma\})$ contains $T_z$.
Since the action of a Fuchsian group is amenable and topologically free, we have the following result, which could be proved by Proposition \ref{two points}, similar to Theorem 3.1 in \cite{Ju07-1}.
\begin{corollary}
Let $\Gamma$ be a Fucsian group and the orbit $\{\varphi(0):\varphi \in \Gamma \}$ consists of at least two linear independent points over $\mathbb{R}$. Then
there is an exact sequence
$$0\rightarrow \mathfrak K\rightarrow C^*(\{C_{\varphi}:\varphi \in \Gamma \})\rightarrow C(\mathbb{T})\rtimes \Gamma \rightarrow 0.$$
\end{corollary}

As an example of this situation, let $\Gamma$ be a Fuchsian group and contains a hyperbolic map (a linear-fractional map with two distinct fixed points on the unit circle) whose fixed points are not diagonal (endpoints of one diameter).
Then the $C^*$-algebra $C^*(\{C_{\varphi}:\varphi \in \Gamma\})$ contains $T_z$. This follows from the fact that fixed points of hyperbolic
and parabolic maps are contained in the limit set.
\begin{corollary} \label{gamma}
Let $\varphi$ is an automorphism of the unit disk. If $\varphi^n(0)$ and $\varphi^m(0)$ are linearly independent over $\mathbb{R}$, for some $n,m \in \mathbb{Z}$, then the unital $C^*$-algebra $C^*(C_{\varphi})$ contains the shift operator.
\end{corollary}
\begin{proof}
Since $C_{\varphi^{-1}}$ is the inverse of $C_{\varphi}$, the $C^*$-algebra $C^*(C_{\varphi})$ contains $C_{\varphi^n}$, for all $n\in \mathbb{Z}$. Now the result follows from Proposition \ref{two points}.
\end{proof}

Let $\varphi \in Aut(\mathbb{D})$ be of the form $\varphi(z)=\omega \frac{s-z}{1-\bar{s}z}$
for some $\omega \in \mathbb{T}$ and $s \in \mathbb{D}$. If $\omega$ is not real $(\omega \neq \pm 1)$
 and $s \neq 0$ then $T_z \in C^*(C_{\varphi})$. Indeed, a simple calculation shows that
 $$\varphi(0)=\omega s,~~~~~~~\varphi^{2}(0)=\omega s \frac{1-\omega}{1-|s|^2\omega}$$
 and if $\omega=x+iy$ $(y\neq0)$, then
 $$\frac{1-\omega}{1-|s|^2 \omega}=\frac{1+|s|^2(1-x)+i(|s|^2 y-y)}{|1-|s|^2 \omega|^2}.$$
 Therefore $\frac{1-\omega}{1-|s|^2 \omega}$ is not real and $\varphi(0)$ and $\varphi^2(0)$ are linearly independent over $\mathbb{R}$. Hence Corollary \ref{gamma} shows that $T_z \in C^*(C_\varphi)$. If $\omega$ is real (1 or -1) or $s\neq 0$, then all $\varphi^{n}(0)$'s are dependent.

 Jury in \cite{Ju07-2} finds the $C^*$-algebra $C^*(T_z,C_{\varphi})/\mathfrak K$, for $\varphi \in Aut(\mathbb{D})$, as a crossed product $C^*$-algebra.
 We do the same when the shift operator is replaced by a general irreducible Toeplitz operator $T_\psi$. The above example  shows that if $\varphi \in Aut(\mathbb{D})$
be of the form (\ref{Auto}) for some non-real $\omega \in \mathbb{T}$ and non-zero $s \in \mathbb{D}$, then the structure of
$C^*(T_z,C_{\varphi})/\mathfrak {K}=C^*(C_{\varphi})/\mathfrak{K}$ does not change, if one replaces $T_z$ with $T_\psi$. Here we check the case $s=0$.
\begin{theorem} \label{exact}
 Let $T_{\psi}$ is irreducible and $\varphi \in Aut(\mathbb{D})$. If the composition operator $C_{\varphi}$  is unitary and
 $\varphi(\psi(\mathbb{T}))=\psi(\mathbb{T})$, then there is an exact sequence of $C^*$-algebras
 $$0\rightarrow \mathfrak K\rightarrow C^*(T_{\psi},C_{\varphi})\rightarrow C(\psi(\mathbb{T})))\rtimes_{\varphi}\mathbb{Z} \rightarrow 0.$$
 \end{theorem}
 \begin{proof}
The image $X=\psi(\mathbb{T})$ of $\psi$ is a compact Hausdorff space and $\varphi^n(X)=X$, for all $n \in \mathbb{Z}$. Now $\mathbb{Z}$ acts on $X$
by
$$\beta:\mathbb{Z}\rightarrow Home(X)$$
$$n\mapsto \beta_{n},~~~~\beta_n(x)=\varphi^n (x),$$
for  $n \in \mathbb{Z}$ and $x \in X$. This induces an action of $\mathbb{Z}$ on $C(X)$ given by
$$\alpha :\mathbb{Z}\rightarrow Aut(C(X)) $$
$$\alpha_n(f)(x)=f(\varphi^{-n}(x)).$$
The $C^*$-algebra $C^*(T_{\psi},C_{\varphi})/\mathfrak K$ is generated by $C^*(T_{\psi})/\mathfrak K\cong C(X)$ and
unitaries $[C_{\varphi^n}]$. On the other hand, Theorem \ref{unitary} shows that the unitary representation $n\rightarrow [C_{\varphi^{-n}}]$
satisfies the covariance relation $[C_{\varphi}]f[C_{\varphi}^*]=\alpha_{n}(f)$. Hence there is a surjective $*$-homomorphism from the full crossed product
$C(X)\rtimes_{\varphi}\mathbb{Z}$ to $C^*(T_{\psi},C_{\varphi})/\mathfrak K$. But  the action of the amenable group $\mathbb{Z}$
on compact Hausdorff space $X$ is amenable and topologically free (i.e. for each $n \in \mathbb{Z}$, the set of points that are fixed by $\varphi^n$ has empty interior) thus similar to the proof of Theorem 2.1 in \cite{Ju07-2}, the above $*$-homomorphism is also injective and hence an isometry.
 \end{proof}
\begin{corollary}
Let $\varphi \in Aut(\mathbb{D})$ be of the form (\ref {Auto}) with $s=0$. If $T_\psi$ is irreducible and
$\varphi(\psi(\mathbb{T}))=\psi(\mathbb{T})$, then there is an exact sequence of $C^*$-algebras
$$0\rightarrow \mathfrak K\rightarrow C^*(T_{\psi},C_{\varphi})\rightarrow C(\psi(\mathbb{T})))\rtimes_{\varphi}\mathbb{Z}\rightarrow 0.$$
\end{corollary}
\begin{proof}
In this case $\varphi (z)=\omega z$, for some $\omega \in \mathbb{T}$. It is easy to check that
$C_{\varphi}$ is an isometry  from $H^2$ onto $H^2$. Hence by the above theorem, $C^*(T_{\psi},C_{\varphi})/\mathfrak K$ is
isomorphic to the crossed product $C(\psi(\mathbb{T})))\rtimes_{\varphi}\mathbb{Z}$.
\end{proof}

One may wonder if for $\varphi \in Aut(\mathbb{D})$ of the form (\ref {Auto}) with $s=0$, there is a function $\psi$ that satisfies in the hypothesis of the above corollary. For example, let $\varphi(z)=ze^{i\frac{2\pi}{3}}$. Theorem 1 in \cite{no67} gives a sufficient condition for irreducibility
of Toeplitz operator $T_{\psi}$: if the restriction of a function $\psi \in H^2$ to a Borel subset $S\subseteq \mathbb{T}$,
with positive normalized Haar measure on the unit circle, is one-to-one and the sets $\psi(S)$ and $ \psi(\mathbb{T}\setminus S)$  are disjoint,
then $T_{\psi}$ is irreducible. By the Riemann mapping theorem (for example, see \cite{gr94}) there exists a biholomorphic (bijective and holomorphic)
map $\psi$ from the unit disk $\mathbb{D}$ onto the simply connected set
$$\mathbb{D}-([1/2,1)\cup [-1/4+\sqrt{3}/4i,-1/2+\sqrt{3}/2i)\cup [-1/4-\sqrt{3}/4i,-1/2-\sqrt{3}/2i))$$ in the complex plane. A result of Carath\'{e}odory in \cite{ca13} states that $\psi$ extends continuously to  the closure of the unit disk. Moreover one can choose $\psi$ (see Figure \ref{fig1}) such that
$$\psi(A)=A^{'},\psi(B)=B^{'},\psi(C)=C^{'},\psi(K)=K,\psi(K^{'})=K{'},\psi(K^{''})=K^{''},$$
and
$$\psi(D)=\psi(D^{'})=A,\psi(E)=\psi(E^{'})=B,\psi(F)=\psi(F^{'})=C.$$
where capital letters show points in the closure of the unit disk.
The restriction of $\psi$ to $\mathbb{T}$, also denotes by $\psi$, is an element of $H^2$, one-to-one on
$$S=\{e^{i\theta}:\theta \in (\frac{2\pi}{15},\frac{8\pi}{15})\cup(\frac{4\pi}{5},\frac{6\pi}{5})\cup(\frac{22\pi}{15},\frac{28\pi}{15})\},$$
$\psi(S)$ and $\psi(\mathbb{T}\setminus S)$ are disjoint and the image of $\psi$ is invariant under $\varphi$, that
is, $\varphi(\psi(\mathbb{T}))=\psi(\mathbb{T})$).
Moreover since $\psi$ is not one-to-one on $\mathbb{T}$, $C^*(T_{\psi})\neq C^*(T_{z})$.
\begin{figure}[h]
 \center
\includegraphics[angle=0,width=.7\textwidth]{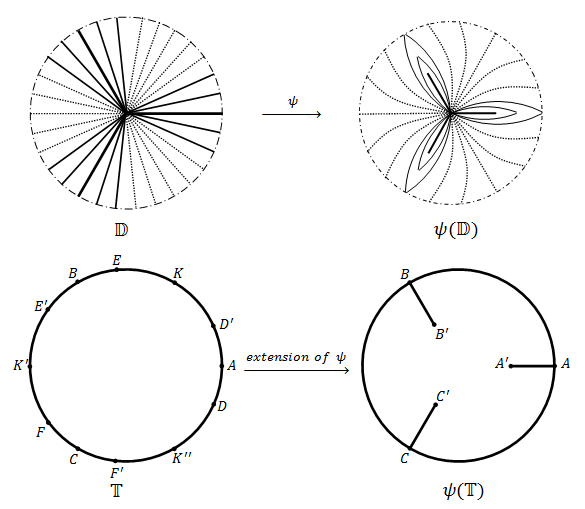}
\caption{\label{fig1}  }
\end{figure}
\begin{figure}[h]
 \center
\includegraphics[angle=0,width=.3\textwidth]{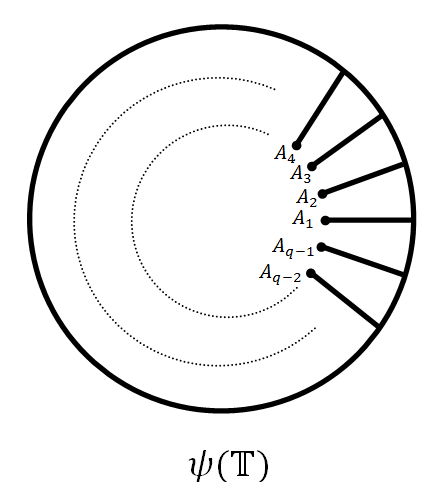}
\caption{\label{fig2}  }
\end{figure}

More generally, if the automorphism $\varphi$ is of the form $\varphi(z)=ze^{i \frac{2p }{q}\pi}$ where  $p$ and $q$ are relatively prime integers with $q$ positive, then by a similar construction, there is a function $\psi$ that satisfies the conditions of the above corollary and is not one-to-one on the unit circle and $$\psi(\mathbb{T})=\mathbb{T}\cup (\bigcup_{n=0}^{n=q-1}\varphi^{n}([1/2,1)),$$
(see Figure \ref{fig2}). We show that the $C^*$-algebras $C(\psi(\mathbb{T}))\rtimes_{\varphi}\mathbb{Z}$ and $C(\mathbb{T})\rtimes_{\varphi}\mathbb{Z}$ are non isomorphic. This is of course easy when $p=0$, that is, $\varphi(z)=z$ on $\mathbb{T}$ and the action of $\mathbb{Z}$ on $\mathbb{T}$ and $\psi(\mathbb{T})$ is trivial. In this case,   $$C(\psi(\mathbb{T}))\rtimes_{id}\mathbb{Z}\cong C(\psi(\mathbb{T}))\otimes C^*(\mathbb{Z})\cong C(\psi(\mathbb{T}))\otimes C(\mathbb{T})\cong C(\psi(\mathbb{T}) \times \mathbb{T})$$
and
$C(\mathbb{T})\rtimes_{id}\mathbb{Z}\cong C(\mathbb{T}\times\mathbb{T}).$
 The spectrum of these $C^*$-algebras are $\psi(\mathbb{T})\times \mathbb{T}=(\mathbb{T}\cup [1/2,1))\times \mathbb{T}$ and $\mathbb{T}^2$, which are not homeomorphic: if, on the contrary,   $\Phi$ is a homeomorphism from $(\mathbb{T}\cup [1/2,1))\times\mathbb{T}$ onto $\mathbb{T}^2$ and $S$ is the half of $\mathbb{T}$ on the left side of the imaginary axis, then  $\Phi(S\times \mathbb{T})=X \times Y$ is connected, and so are $X$ and $Y$. If $X\neq \mathbb{T}$ and $Y\neq\mathbb{T}$ then $\mathbb{T}\setminus X$ and $\mathbb{T}\setminus Y$ are connected subsets of $\mathbb{T}$. Hence $((\mathbb{T}\setminus X)\times \mathbb{T})\cup (\mathbb{T}\times (\mathbb{T}\setminus Y))\subseteq \mathbb{T}^2$ is connected, and so is $\Phi^{-1}(((\mathbb{T}\setminus X)\times \mathbb{T})\cup(\mathbb{T}\times(\mathbb{T}\setminus Y)))=(\psi(\mathbb{T})\setminus S)\times \mathbb{T}$, a contradiction, as $\psi(\mathbb{T})\setminus S$ is not  connected. Similarly $X=\mathbb{T}$ or $Y=\mathbb{T}$ lead to a contradiction.

For the more general example, we need some preparation to show that the crossed products are non isomorphic (we refer the reader to \cite{wi07} for more details).
Let $Y$ is  a topological space. The $T_0$-ization of $Y$ is the quotient space $(Y)^{\sim}=Y/\sim$ where $\sim$ is the equivalence relation on $Y$ defined by $x\sim y$ if $\overline{\{x\}}=\overline{\{y\}}$. The space $(Y)^{\sim}$, equipped with the quotient topology, is a $T_0$ topological space (see Lemma 6.10 in \cite{wi07}).

Let $G$ be a topological group acting on a topological space $X$ from left. The orbit of  $x \in X$ is the set $G\cdot x=\{s\cdot x:s\in G\}$. The stability group at $x$ is $G_x:=\{s\in G:s\cdot x=x\}$. The set of orbits is denoted by $G\backslash X$ and is called the orbit space. It is equipped with the weakest topology making the natural quotient map $p:X\rightarrow G\backslash X$ continuous.

By Lemma 3.35 in \cite{wi07}, if $X$ is second countable or locally compact, then so is $G\backslash X$. When $G$ is a locally compact abelian group and $\widehat{G}$ is the character group of $G$, one could define an equivalence relation on $X\times \widehat{G}$ by  $(x,\tau)\sim (y,\sigma)$ if $\overline{G\cdot x}=\overline{G\cdot y}$ and $\tau\bar{\sigma} \in G_x^{\perp}$. The space $X\times \widehat{G}/\sim$ is given the quotient topology. Note that if $X$ is a $T_0$-space, then $X\times \widehat{G}/\sim=X\times \widehat{G}/G_x^{\perp}$. By Remark 8.42 in \cite{wi07}, $X\times \widehat{G}/\sim$ is homeomorphic to $(G\backslash X)^{\sim}\times \widehat{G}/ \sim.$

\begin{proposition}
Let $\psi$ be as above and $\varphi(z)=ze^{i \frac{2p }{q}\pi}$, with $p$ and $q$ relatively prime integers and $q>0$. Then the spectrum of the $C^*$-algebra $C(\psi(\mathbb{T}))\rtimes_{\varphi}\mathbb{Z}$ is homeomorphic to $(\mathbb{T}\cup [1/2,1))\times \mathbb{T}$.
\end{proposition}
\begin{proof}
Consider the map $$f:\mathbb{Z}\backslash \psi(\mathbb{T})\rightarrow \mathbb{T}\cup [1/2,1)$$
\begin{equation*}
\mathbb{Z}\cdot z \mapsto \left\{
\begin{array}{rl}
z^q& \text{if}~~~z\in \mathbb{T},\\
z^q/|z|^{q-1} & \text{otherwise.}
\end{array} \right.
\end{equation*}
Since the map $$g:\psi(\mathbb{T})\rightarrow \mathbb{T}\cup [1/2,1)$$
\begin{equation*}
 z \mapsto \left\{
\begin{array}{rl}
z^q& \text{if}~~~z\in \mathbb{T},\\
z^q/|z|^{q-1} & \text{otherwise}
\end{array} \right.
\end{equation*}
is continuous, $g^{-1}(U)$ is open in $\psi(\mathbb{T})$ for each open set $U$ in $\mathbb{T}\cup [1/2,1)$. Since the quotient map $p:\psi(\mathbb{T})\rightarrow \mathbb{Z}\backslash \psi(\mathbb{T})$  is open (Lemma 3.25 in \cite{wi07}), $f^{-1}(U)=p(g^{-1}(U))$ is open in $\mathbb{Z} \backslash \psi(\mathbb{T})$ and so $f$ is continuous. On the other hand, $\mathbb{Z}\backslash \psi(\mathbb{T})$ is compact, therefore $f$ is a homeomorphism. Also $\mathbb{Z}\backslash \psi(\mathbb{T})$ is $T_0$ and $\mathbb{Z}$ and $\psi(\mathbb{T})$ are second countable, hence by Theorem 8.39 and Remark 8.42 in \cite{wi07},
$$Prim(C(\psi(\mathbb{T}))\rtimes_{\varphi}\mathbb{Z})\cong (\mathbb{Z}\backslash \psi(\mathbb{T}))^{\sim}\times\widehat{\mathbb{Z}}/\sim \cong \mathbb{Z}\backslash \psi(\mathbb{T})\times \widehat{\mathbb{Z}}/(q\mathbb{Z})^{\perp}$$
 $$\cong (\mathbb{T}\cup [1/2,1))\times \widehat{q\mathbb{Z}}\cong (\mathbb{T}\cup [1/2,1))\times \mathbb{T}.$$
  But all $\mathbb{Z}$-orbits $\mathbb{Z}\cdot z$ are finite (and so closed in $\psi(\mathbb{T})$), therefore by Theorem 8.44 in \cite{wi07}, $C(\psi(\mathbb{T}))\rtimes_{\varphi}\mathbb{Z}$ is a liminal $C^*$-algebra and by Theorem 5.6.4 in \cite{mu90}, $Prim(C(\psi(\mathbb{T}))\rtimes_{\varphi}\mathbb{Z})$ is homeomorphic to the spectrum $(C(\psi(\mathbb{T}))\rtimes_{\varphi}\mathbb{Z})^\wedge$, as required.
\end{proof}
Now we could finish the argument for the general example. The spectrums of the $C^*$-algebras $C(\psi(\mathbb{T}))\rtimes_{\varphi}\mathbb{Z}$ and $C(\mathbb{T})\rtimes_{\varphi}\mathbb{Z}$
are $\psi(\mathbb{T})\times \mathbb{T}$ and $\mathbb{T}^2$, which are not homeomorphic. Hence these $C^*$-algebras are not isomorphic.


\end{document}